\documentclass[a4paper,12pt]{article}
\usepackage[utf8]{inputenc}
\usepackage{amsmath}
\usepackage{amssymb,amscd}
\usepackage{amscd}
\usepackage{amsmath,latexsym,amssymb,amsthm}

\usepackage{txfonts}
\usepackage{tikz}
\usepackage{bbm}
\usepackage{hyperref}
\usepackage{cite}

\theoremstyle{plain}
\newtheorem{thm}[subsection]{Theorem}
\newtheorem{lem}[subsection]{Lemma}
\newtheorem{prop}[subsection]{Proposition}
\newtheorem{cor}[subsection]{Corollary}
\newtheorem{question}{Question}

\newtheorem{conj}[subsection]{Conjecture}

\newcommand{\card}[1]{\ensuremath{|#1|}}
\providecommand{\keywords}[1]{\textbf{\textit{Keywords:}} #1}
\providecommand{\classification}[1]{\textbf{\textit{AMS subject classification:}} #1}
\theoremstyle{definition}

\newtheorem{rem}[subsection]{Remark}
\newtheorem{defn}[subsection]{Definition}

\def\a{\alpha}

\def\b{\beta}

\def\lesc{\preccurlyeq}

\def\a{\alpha}

\def\b{\beta}

\title{Beck's Conjecture for Power Graphs}
\author{Priya Das\footnote{priya.math88@gmail.com}, ~ Himadri Mukherjee \footnote{himadri@iiserkol.ac.in} \,\footnote{ The authors thank Mr. Suman Bandyopadhyay for many helpful discussions and suggestions.}}
\date{}
\begin{document}

\maketitle
\setlength{\parindent}{0pt}

\begin{abstract} Beck's conjecture on coloring of graphs associated to various algebraic objects has generated considerable interest in the community of discrete mathematics and combinatorics since its inception in the year 1988. The version of this conjecture for power-graphs of finite groups has been addressed and partially settled by previous authors. In this paper we answer it in the affirmative in complete generality, and, in effect, we establish a ``nicer'' statement on a larger class of graphs. We also clear up certain ambiguities present in the way the previous versions of the conjecture were posed.

\end{abstract}
\\
\classification{05C25 Primary, 05C17 Secondary}

\section{Introduction} Zero divisor graph of a commutative ring was introduced in the year 1988 by I. Beck in the article \cite{beck}, in which he related interesting commutative algebraic notions to coloring and clique numbers of the zero divisor graph. Beck conjectured that the chromatic number of the zero divisor graph of a commutative ring is same as its clique number. He called a ring with this property a chromatic ring. So the conjecture can be restated as: every commutative ring is a chromatic ring. Beck's conjecture for chromatic rings has been disproved in general. A minimal example of a non-chromatic ring was given in the article\cite{counter}, and a series of counter-examples have appeared since (viz. \cite{bhat}). The question of classifying all the chromatic rings till date remains open. The reader might find the following articles interesting: \cite{survey_beck} for the history of the zero divisor graphs, and \cite{maimani,shaban} for review of known results. The question has been generalized to many other areas of graph theory where a graph is attached to an algebraic object.

On the other hand the power graph of a group, which is a graph structure defined on the set of points of the group by declaring an edge between two elements if one is contained in the cyclic subgroup generated by the other, was introduced in the paper \cite{kelarev}. A good number of exciting results have appeared in\cite{cameron, cameron2, ivy,survey}. The paper\cite{ashrafi} first addressed the problem of coloring in power-graphs and the equality of chromatic number with the clique number.

The authors of \cite{ashrafi} establish a couple of theorems in this direction. The \emph{exponent} of a finite group is the least common multiple of the orders of the elements of the group, and a \emph{full exponent group} is a group $G$ in which there is an element $g$ such that $o(g)$ equals the exponent.
\begin{thm}\cite[theorem 3]{ashrafi}\label{ashrafi}
 Let $G$ be a full exponent group, with exponent $p_1^{\b_1}p_2^{b_2}\ldots p_r^{\b_r}$ and $p_1<p_2\ldots<p_r$ then: \[\omega(G)=\chi(G)=p_r^{\b_r}+\displaystyle{\sum_{j=0}^{r-2} (p_{r-j-1}^{\b_{r-j-1}}-1)\prod_{i=0}^{j} \phi(p_{r-i}^{\b_{r-i}})}.\]
\end{thm}
\cite[Theorem 2]{ashrafi} deals with the same result for a cyclic group, and the above result is inferred by appealing to similarity. For the cyclic group case, the main focus has been to show $\omega(G)=p_r^{\b_r}+\displaystyle{\sum_{j=0}^{r-2} (p_{r-j-1}^{\b_{r-j-1}}-1)\prod_{i=0}^{j} \phi(p_{r-i}^{\b_{r-i}})}$ (which we will call the \emph{first equality} for the rest of this section) whilst the \emph{second equality}, namely $\omega(G)=\chi(G)$, is mentioned as a direct consequence of the \emph{strong perfect graph theorem} \cite{perfect}. The following conjecture appears in \cite{ashrafi} and has been cited in, eg. \cite{survey}.
\begin{conj}\cite[conjecture 1]{ashrafi}\label{conj}
 The theorem \cite[theorem 3]{ashrafi} is correct in general.
\end{conj}
Ambiguity enters here as either of the equalities or both can be generalized and in that generality the conjecture easily seen to be false, as we see below. The second equality, which is the focus of this paper, generalizes as is. However, applying the strong perfect graph theorem does require some argument, which does not seem to call for any assumption related to cyclicity and full exponentiality of the group. In all, we are inclined to dismiss the relevance of full exponent groups in this context.

The second part, in fact is easier to see. Note that the power graph being the underlying undirected graph of the directed edges $g \rightarrow h$ if $g^k=h$ for some $k$, which in fact forms a preorder, a maximal clique is induced by a directed path of maximum length. In this context, in fact it will consist of a collection of  powers of the source of this path, say $g$, and hence be contained in the cyclic subgroup generated by $g$.

Let $n=p_1^{\a_1}p_2^{\a_2}\ldots p_k^{\a_k}$ be the order of $g$ with $p_k > p_{k-1} \ldots > p_1$, define $\Psi(n)$ to be the size of the largest clique in $\langle g \rangle$. Since the longest directed path must visit all the $\phi(n)$ generators of $\langle g \rangle$ starting at one of them and then extend to the $p^{th}$ power for a prime $p | n$, we see the recurrence $\Psi(n)= \phi(n)+ \mathrm{max}_{p | n}\{\Psi(\frac{n}{p}) \}$. In fact it can be resolved to the following formula: 

\label{formula}
  \[\Psi(n)=1+\displaystyle{\sum_{r=1}^{k}(p_{r}^{\a_r}-1)\prod_{j=r+1}^{k}p_{j}^{\a_j-1}(p_j-1)}\]
Note that this concurs with the formula in \cite[theorem 3]{ashrafi}. The above discussion can be summarized in the following simplified proposition:
\begin{prop}\label{computation}
 Let $g \in G$ be such that $\Psi(o(g))$ is maximum then $\omega(G)=\Psi(o(g))$
\end{prop}

Now it is easy to see why the straight-forward generalization of the first equality need not be true in general, as the maximum of $\Psi(o(g))$ need not coincide with $\Psi(n)$ for the exponent $n$ of $G$. For example if $G=S_5$ we see that $n =60$, $\Psi(n)=37$ but $\omega(G)=5$. In our main theorem stated below we prove the conjecture \ref{conj}, with the said correction implemented. 
\begin{thm}\label{main}
 Let $G$ be the power graph of a finite group, and let $g \in G$ be such that $\Psi(o(g))$ is maximum, then $ \chi(G)=\omega(G)=\Psi(o(g))$.
\end{thm}

 We give a proof of the above in the generality of preordered sets in section 3. Attempting to give a more elementary argument, however, we encounter an interesting stability property of the colorings of power graphs, namely that the coloring restricts to a minimal coloring on any subgroup. Here we make a definition: 
 \begin{defn}\label{stable}
 For a finite group $G$ a \emph{stable coloring} on $G$ is a coloring which restricts, for any subgroup $H \subset G$, to a coloring on $H$ with $\chi(H)$ colors. The coloring is said to be \emph{weakly stable} if it holds only for cyclic subgroups $H$.
\end{defn}
We in fact show that the power graph of a finite group admits a weakly stable coloring in section 4. Moreover, we show that for a cyclic group $G$ a stable coloring on a subgroup extends to a stable coloring on $G$. Both these statements are possibly generalizable to all groups and stable colorings. In addition this gives rise to purely graph theoretic questions (with the notion of stability generalized to subgraphs): 
\begin{question}
 Which classes of graphs admit stable colorings? 
\end{question}
\begin{question}
 Does the extension property hold true for stable colorings on Berge graphs?
\end{question}


\section{The Basics}
In this section we gather few elementary definitions and results which can be found in \cite{cameron,cameron2,survey, harary}. A directed power graph of a group $G$ is a graph $P_G=(G,E)$ where $(g,h) \in E \mbox{ if } \exists k \in \mathbb{N}, \mbox{ such that } g^k=h$. Note that if $(g,h) \in E$ and $(h,l) \in E$ then $(g,l) \in E$. So the directed power graph can be thought of as a pre-ordered set if we define $g \preccurlyeq h $ whenever $(g,h) \in E$. A graph $G$ is called a Berge graph if there are no holes or antiholes of odd length\cite{perfect}, where a hole is a set of vertices of size more than or equal to four in the graph $G$ whose induced subgraph is a cycle. And an antihole is a hole in the complement graph. And a graph is called perfect if for every induced subgraph $H$ of the graph, the chromatic number equals the size of the largest clique in $H$. Let us also recall the strong perfect graph theorem from the article\cite{perfect}:
\begin{thm}\cite{perfect}\label{perfect} A graph is a Berge graph if and only if it is perfect.
 
\end{thm}

\section{The First Proof}

\begin{thm}\label{berge} The underlying undirected graph of a preordered set $(V, \lesc)$ is a Berge graph.
 
\end{thm}
\begin{proof}
 follows from lemmas \ref{hole} and \ref{antihole}
\end{proof}

\begin{cor}
 Theorem \ref{main} follows.
\end{cor}

Let us denote, by abuse of notation, the underlying undirected graph of a preordered set $(V,\lesc)$ with $V=(V,E)$.
\begin{lem}\label{hole}
 Let $(V,\lesc)$ be a preordered set. Then there is no hole of odd length in the graph $V$.
\end{lem}
\begin{proof} Let $v_0,v_1,v_2 \ldots v_{2n}$, $v_i$ distinct for distinct $i$, be a hole in $V$ i.e. $(v_i,v_{i+1}) \in E \mbox{ and } (v_{2n},v_0) \in E \mbox{ and } (v_i, v_j) \notin E \mbox{ else }$. If $v_i \lesc v_{i+1}$ then let us call the edge $(v_i,v_{i+1})$ a red edge, else let us call it a blue edge. Note that if two subsequent edges $(v_i,v_{i+1}), (v_{i+1}, v_{i+2})$ have the same color then we have $(v_i, v_{i+2}) \in E$ so the colors must alternate. But since there are odd number of edges in the hole, there must be two consecutive edges of same color. A contradiction.  
 
\end{proof}

\begin{lem}\label{zero}
 For any $i, j , j' $, if $ (v_i,v_j), (v_i, v_{j'}) \in E \mbox{ and } (v_j, v_{j'}) \notin E  $, then, either $v_i \lesc v_j \mbox { and } v_i \lesc v_{j'}$, or, $ v_j \lesc v_i \mbox{ and } v_{j'} \lesc v_i$.
\end{lem}
\begin{proof}
 Else $v_j \lesc v_i \lesc v _{j'}$ or $v_{j'} \lesc v_i \lesc v_j$. Either way $(v_j,v_{j'}) \in E$, a contradiction.
\end{proof}

\begin{lem}\label{first}
 Let $v_0,v_1,\ldots, v_{2n}$ be an antihole. Then for any $i \, (0 \leq i \leq 2n)$ if $v_j \lesc v_i$ for some $j$ such that $(v_i,v_j) \in E$, then $v_j \lesc v_i \; \forall \, j \mbox{ such that } (v_i,v_j) \in E$ 
\end{lem}
\begin{proof}
Without loss of generality by possible renumbering of the vertices let us assume that $i=2$, then $(v_i,v_j) \in E$ precisely for $j=3,4, \ldots , 2n$. Note that $v_3 \lesc v_2 , \, v_4 \lesc v_2, \ldots v_k \lesc v_2 \Rightarrow v_{k+1} \lesc v_2$ too by lemma \ref{zero} since $(v_k,v_{k+1}) \notin E$. So inductively $v_j \lesc v_2 \, \forall j =3,4, \ldots 2n$ if $v_3 \lesc v_2$. Similarly we argue the case $v_3 \succcurlyeq v_2$.
\end{proof}

\begin{lem}\label{antihole}
 Let $(V,\lesc)$ be a preordered set. Then there is no antihole in the underlying undirected graph $V$ of odd length.
\end{lem}
\begin{proof}
 Let $v_0,v_1, \ldots , v_{2n}$, $v_i \ne v_j \mbox{ for } i \ne j$, be an antihole, so $(v_i, v_{i+1}) \notin E$ and  $(v_{2n}, v_0) \notin E$. Then by lemma \ref{first} for any $v_i$ we have either $v_i \lesc v_j$ for every $j$ such that $(v_i,v_j) \in E$, we color these nodes red, else we have $v_j \lesc v_i$ for all $j$ such that $(v_i,v_j) \in E$ we color such nodes blue. Note that by lemma \ref{first} this is well defined i.e. all the vertices can be labeled by these two colors. Now since the induced subgraph on $v_0,v_1,v_2 , \ldots v_{2n}$ is regular, say of degree be $d$, we have the number of incident edges on red vertices =$d\, \card{ \{v_i \, | \,  v_i \mbox{ is colored red } \}}$ and the number of edges incident on the blue vertices = $d\,\card{\{v_i \, | \, v_i \mbox{ is colored blue } \} }$. Since any edge is incident on a red and a blue vertex we have $\card{ \{v_i \, | \, v_i \mbox{ is colored red } \}}=\card{\{v_i \, | \, v_i \mbox{ is colored blue } \} }$. But that there are odd number of vertices makes this scenario impossible.
\end{proof}

Now the equality $\omega(G)=\chi(G)$ follows from \ref{perfect} and the fact that the graph is perfect since it is Berge by \ref{berge}.

\section{The Second Proof}
In this section we will give a more elementary (that does not depend on the strong perfect graph theorem) argument to prove the main theorem \ref{main}, but first we have to gather a few elementary lemmas.

\begin{lem}\label{cyclic} Given a subgroup $H$ of $G=\mathbb{Z}/n\mathbb{Z}$ any stable coloring on the subgroup extends to a stable coloring on the group $G$.
 
\end{lem}
\begin{proof}
 Taking a filtration of the group $G$ such as $H < H_1 < H_2 \ldots H_N=G$ such that $[H_{i+1}: H_i]$ is a prime we reduce without loss of generality to the case that the given subgroup $H$ is of index $p$ in $G$ where $p$ is a prime. Let us assume that $n=p_1^{k_1}p_2^{k_2}p_3^{k_3} \ldots p_r^{k_r}$ where $p_1 < p_2 < p_3 \ldots < p_r$ are primes, $k_i >0$ and $p=p_i$ for some $i$. Let us also say that we have a stable coloring on the subgroup $H$. By the lemma \ref{counting} we have a coloring on $G $ with $\chi(G) - \chi(H)$ extra colors, and since we have used $\chi(H)$ colors to color $H$ we have given a coloring on $G$ extending the coloring on $H$ with $\chi(G)=\Psi(n)$ colors (as in proposition \ref{computation}), and thus a minimal coloring on $G$. To show the stability let us take an element $x \in G$, and note that by the assumption that it is a stable coloring on $H$ it restricts to a minimal coloring on $H \cap \langle x \rangle$. Since the process of \ref{counting} restricts identically to $\langle x \rangle \cap H < \langle x \rangle  $, we see that it is the minimal coloring on the subgroup $\langle x \rangle$. Hence the coloring on $G$ is stable.
\end{proof}

\begin{lem} \label{counting} Let $H,G$ be as above, then any coloring on $H$ can be extended to a coloring on $G$ with $\chi(G) - \chi(H)$ colors.
 
\end{lem} 
\begin{proof}
Let us assume that $n=p_1^{k_1}p_2^{k_2} \ldots p_r^{k_r}$ where $p=p_i$ following the notation of the above lemma. The subgroup generated by the element $g^{p_1^{k_1-j_1}p_2^{k_2-j_2}p_3^{k_3-j_3} \ldots p_r^{k_r-j_r}}$ uniquely corresponds to the vector $(j_1,j_2,j_3, \ldots ,j_r)$ where $G=\langle g \rangle $. Such subgroups are generated by $\phi(p_1^{j_1}p_2^{j_2}\ldots p_r^{j_r})$ generators which form a clique of that size. 
In this way $(0,0,\ldots, 0)$ corresponds to the trivial subgroup, $(k_1,k_2,\ldots, k_r)$ corresponds to $G$, and $(k_1,k_2, \ldots, k_i-1, \ldots, k_r)$ corresponds to the subgroup $H$. Note also that $(j_1,j_2, \ldots, j_r) \leq (j_1',j_2', \ldots, j_r')$ corresponds to the inclusion of subgroups.

The elements in $G \setminus H$ are precisely those generating a subgroup corresponding to $(j_1,j_2, \ldots, j_r)$ where $j_i =k_i$, that is $(j_1,j_2, \ldots, j_r)$ on the $j_i=k_i$ wall.  We color these by reusing the colors on the points $(j_1',j_2' , \ldots, j_r')$ inside $H$, given by the map $(j_1,j_2, \ldots, j_r) \longmapsto (j_1', j_2' \ldots j_r')$, where $j_i' \coloneqq k_i -1$ and for $t \neq i$, $j_t' \coloneqq j_t+1$ for $t=\mathrm{ min \, }\{u \, | \, j_u \neq k_u \}$,  $j_t'\coloneqq j_t$ otherwise. Clearly $\underline{j}$ and $\underline{j'}$ are not comparable. This works for $(j_1,j_2, \ldots, j_r) \neq (k_1, k_2, \ldots, k_r)$ for which we will need $\phi(p_1^{j_1}p_2^{j_2}\ldots, p_r^{j_r})- \phi(p_1^{j_1'}p_2^{j_2'}\ldots, p_r^{j_r'})$ extra colors at most; and $\phi(p_1^{k_1}p_2^{k_2} \ldots, p_r^{k_r})$ colors for the generators of $G$. From the formula of $\Psi(n)=\omega(G)$ \, (see \ref{computation}) it is clear that we will be using exactly $\Psi(n) - \Psi(n/p)= \chi(G) - \chi (H)$ extra colors.   
\end{proof}

\begin{thm}\label{main2}
 For any finite group $G$, $\chi(G)=\omega(G)$.
\end{thm}
\begin{proof}
Choose $g \in G$ such that $\Psi(o(g))$ is maximum and denote the maximum number by $\mu$. Since we know that $\omega(G) \leq \chi(G)$ it will be sufficient to show a coloring with $\mu$ colors. We show this by inductively constructing subsets $S_0 \subset S_1 \subset S_2 \subset \cdots \subset S_m= G$  such that each $S_i$ is upward closed in the preorder, with a weakly stable coloring $c_i$ on $S_i$ using $\mu$ colors; i.e. for any $h \in S_i$, $c_i |_{\langle h \rangle}$ is a minimal coloring on $\langle h \rangle $. Choose $S_0= \langle g \rangle $ and color $S_o$ with $\mu$ colors satisfying the above requirements using the lemma \ref{cyclic}. 
Having constructed $S_i$, if $S_i \neq G$, construct $S_{i+1}$ by choosing $h \in G \setminus S_i$ and setting $S_{i+1}= S_i \cup \langle h \rangle$. Note that $S_i \subsetneq S_{i+1}$ and $S_{i+1} $ is upward closed by construction. Now $S_i \cap \langle h \rangle $ is upward closed, hence it is a subgroup of $\langle h \rangle $, let it be generated by $h '$. Now as by assumption $c_i|_{ \langle h' \rangle}$ is stable, by the lemma \ref{cyclic} we can extend it to a stable coloring $c'$ on $\langle h \rangle$. Since $c'|_{\langle h' \rangle}=c_i|_{\langle h' \rangle}$ they patch up to give a labeling on all of $S_{i+1}$, which is in fact a coloring since there are no edges between $S_i \setminus \langle h \rangle $ and $\langle h \rangle \setminus S_i$. Further, this uses $\chi(\langle h \rangle ) - \chi( \langle h' \rangle)$ extra colors, so we can relabel $c'$ by reusing the $\mu - \chi(\langle h' \rangle)$ colors from $S_i \setminus \langle h' \rangle$, on the part $\langle h \rangle \setminus S_i$ to get a coloring $c_{i+1}$ on $S_{i+1}$ using $\mu$ colors ($\chi(\langle h \rangle) \leq \mu$ by assumption). Finally $c_{i+1} $ is stable as, for $x \in S_{i+1}$, $c_{i+1}$ restricts to a minimal coloring on $\langle x \rangle$: inductively if $x \in S_i$, and by construction if $x \in \langle h \rangle$ since $c_{i+1}|_{\langle h \rangle} $ is a relabeling of the stable coloring $c'$.
\end{proof}

\begin{rem}
 In fact the process above yields a weakly stable coloring on $G$ (cf. \ref{stable}).
\end{rem}

It appears that one can reformulate this argument using elementary homomorphisms for a preordered set, maintaining stability on its ideals.
\bibliographystyle{abbrv}
\bibliography{beck}

\begin{thebibliography}{10}

\bibitem{survey}
J.~Abawajy, A.~Kelarev, and M.~Chowdhury.
\newblock Power graphs: a survey.
\newblock {\em Electron. J. Graph Theory Appl. (EJGTA)}, 1(2):125--147, 2013.

\bibitem{counter}
D.~D. Anderson and M.~Naseer.
\newblock Beck's coloring of a commutative ring.
\newblock {\em J. Algebra}, 159(2):500--514, 1993.

\bibitem{shaban}
D.~F. Anderson, S.~Ghalandarzadeh, S.~Shirinkam, and P.~Malakooti~Rad.
\newblock On the diameter of the graph {$\Gamma_{Ann(M)}(R)$}.
\newblock {\em Filomat}, 26(3):623--629, 2012.

\bibitem{beck}
I.~Beck.
\newblock Coloring of commutative rings.
\newblock {\em J. Algebra}, 116(1):208--226, 1988.

\bibitem{bhat}
S.~M. Bhatwadekar, M.~N. Dumaldar, and P.~K. Sharma.
\newblock Some non-chromatic rings.
\newblock {\em Comm. Algebra}, 26(2):477--505, 1998.

\bibitem{cameron2}
P.~J. Cameron.
\newblock The power graph of a finite group, {II}.
\newblock {\em J. Group Theory}, 13(6):779--783, 2010.

\bibitem{cameron}
P.~J. Cameron and S.~Ghosh.
\newblock The power graph of a finite group.
\newblock {\em Discrete Math.}, 311(13):1220--1222, 2011.

\bibitem{ivy}
I.~Chakrabarty, S.~Ghosh, and M.~K. Sen.
\newblock Undirected power graphs of semigroups.
\newblock {\em Semigroup Forum}, 78(3):410--426, 2009.

\bibitem{perfect}
M.~Chudnovsky, N.~Robertson, P.~Seymour, and R.~Thomas.
\newblock The strong perfect graph theorem.
\newblock {\em Ann. of Math. (2)}, 164(1):51--229, 2006.

\bibitem{survey_beck}
J.~Coykendall, S.~Sather-Wagstaff, L.~Sheppardson, and S.~Spiroff.
\newblock On zero divisor graphs.
\newblock In {\em Progress in commutative algebra 2}, pages 241--299. Walter de
  Gruyter, Berlin, 2012.

\bibitem{harary}
F.~Harary.
\newblock {\em Graph theory}.
\newblock Addison-Wesley Publishing Co., Reading, Mass.-Menlo Park,
  Calif.-London, 1969.

\bibitem{kelarev}
A.~V. Kelarev and S.~J. Quinn.
\newblock Directed graphs and combinatorial properties of semigroups.
\newblock {\em J. Algebra}, 251(1):16--26, 2002.

\bibitem{maimani}
H.~R. Maimani, M.~R. Pournaki, A.~Tehranian, and S.~Yassemi.
\newblock Graphs attached to rings revisited.
\newblock {\em Arab. J. Sci. Eng.}, 36(6):997--1011, 2011.

\bibitem{ashrafi}
M.~Mirzargar, A.~R. Ashrafi, and M.~J. Nadjafi-Arani.
\newblock On the power graph of a finite group.
\newblock {\em Filomat}, 26(6):1201--1208, 2012.

\end{thebibliography}
\end{document}